\UseRawInputEncoding
\documentclass[12pt]{article}

\usepackage{dsfont}
\usepackage{amsfonts,amsmath,amsthm}
\usepackage{algorithm, algorithmic}
\usepackage{color}
\usepackage{graphicx}
\usepackage{stmaryrd}        

\usepackage[paper=a4paper,dvips,top=2cm,left=2cm,right=2cm,
    foot=1cm,bottom=4cm]{geometry}

\begin{document}
\large

\title{Motion, Unit Dual Quaternion and Motion Optimization}
\author{ Liqun Qi\footnote{Department of Mathematics, School of Science, Hangzhou Dianzi University, Hangzhou 310018 China; Department of Applied Mathematics, The Hong Kong Polytechnic University, Hung Hom, Kowloon, Hong Kong
({\tt maqilq@polyu.edu.hk}).}
}
\date{\today}
\maketitle

\begin{abstract}
We introduce motions as real six-dimensional vectors.   A motion means a rotation and a translation. We define a motion operator which maps unit dual quaternions to motions, and a UDQ operator which maps motions to unit dual quaternions.  By these operators, we present the formulation of motion optimization, which is actually a real unconstrained optimization formulation.   Then we formulate two classical problems in robot research, i.e., the hand-eye calibration problem and the simultaneous localization and mapping (SLAM) problem as motion optimization problems.   This opens a new way to solve these problems via real unconstrained optimization.

\medskip



\end{abstract}

\renewcommand{\Re}{\mathds{R}}
\newcommand{\rank}{\mathrm{rank}}
\renewcommand{\span}{\mathrm{span}}
\newcommand{\X}{\mathcal{X}}
\newcommand{\A}{\mathcal{A}}
\newcommand{\I}{\mathcal{I}}
\newcommand{\B}{\mathcal{B}}
\newcommand{\C}{\mathcal{C}}
\newcommand{\OO}{\mathcal{O}}
\newcommand{\e}{\mathbf{e}}
\newcommand{\0}{\mathbf{0}}
\newcommand{\dd}{\mathbf{d}}
\newcommand{\ii}{\mathbf{i}}
\newcommand{\jj}{\mathbf{j}}
\newcommand{\kk}{\mathbf{k}}
\newcommand{\va}{\mathbf{a}}
\newcommand{\vb}{\mathbf{b}}
\newcommand{\vc}{\mathbf{c}}
\newcommand{\vq}{\mathbf{q}}
\newcommand{\vg}{\mathbf{g}}
\newcommand{\vt}{\rm{vec}}
\newcommand{\vx}{\mathbf{x}}
\newcommand{\vy}{\mathbf{y}}
\newcommand{\vu}{\mathbf{u}}
\newcommand{\vv}{\mathbf{v}}
\newcommand{\y}{\mathbf{y}}
\newcommand{\vz}{\mathbf{z}}
\newcommand{\T}{\top}

\newtheorem{Thm}{Theorem}[section]
\newtheorem{Def}[Thm]{Definition}
\newtheorem{Ass}[Thm]{Assumption}
\newtheorem{Lem}[Thm]{Lemma}
\newtheorem{Prop}[Thm]{Proposition}
\newtheorem{Cor}[Thm]{Corollary}
\newtheorem{example}[Thm]{Example}
\newtheorem{remark}[Thm]{Remark}

\section{Introduction}

In \cite{Qi22}, we formulated two classical problems in robot research, i.e., the hand-eye calibration problem \cite{Da99, LWW10, LLDL18, SA89, ZRS94} and
the simultaneous localization and mapping (SLAM) problem \cite{BS07, BLH19, CCCLSNRL16, CTDD15, WJZ13, WND00} as equality constrained dual quaternion optimization problems.  In that formulation, the variables are dual quaternions, while the objective function values are dual numbers.   Then the total order of dual numbers, proposed in \cite{QLY22}, was used for minimization.  Numerical methods for solving the hand-eye calibration problem based on that approach were reported in \cite{CLQY22}.

In the above dual quaternion optimization approach, there is a notable aspect worth being improved.   The variables in the applications such as the hand-eye calibration problem and the SLAM problem are unit dual quaternions.   The equality constraints in such a dual quaternion optimization formulation are nothing else, but the constraints for dual quaternions to be unit dual quaternions.   Note that the sum of two unit dual quaternions is not a unit dual quaternion.  Thus, one possible way is to transform unit dual quaternions to a real linear vector space.    As each unit dual quaternion has only six freedoms, while a dual quaternion has eight freedoms, this may both reduce the dimension of the formulation and get rid of such equality constraints.  Then the problem is formulated as a real unconstrained problem with reduced dimension.  Also, the objective function value here are real numbers.  This also gets rid of the dual number objective function values which are troublesome.

This is the motivation of this paper.  In the next section, we introduce motions.   Each motion is a real six-dimensional vector.  A motion means a rotation and a translation.  Then we define a motion operator, which maps unit dual quaternions to motions, and a UDQ operator, which maps motions to unit dual quaternions.    Some basic properties of such operators are studied.   In Section 3, we present the formulation of motion optimization, which actually is a real unconstrained optimization formulation.   Then we formulate the hand-eye calibration problem and the SLAM problem as motion optimization problems.  Some further properties of the motion operator and the UDQ operator are studied in Section 4.   Final remarks are made in Section 5.

\section{Motions and Unit Dual Quaternions}

Denote the set of real numbers by $\mathbb R$.  We use small Greek letters such as $\alpha, \beta$, or small Roman letters with indices, such as $p_0$, $q_1$ and $t_2^{(i)}$ to denote real numbers.

We use small Roman letters such as $r$ and $t$ to denote real three-dimensional vectors, and small letters with indices $1, 2$ and $3$ to denote their three components.    The letter $r$ denotes the {\bf rotation vector}, while $t$ denotes the {\bf translation vector}.  Sometimes we simply call $r$ a rotation and $t$ a translation.   We have
$r = [r_1, r_2, r_3]$ and $t = [t_1, t_2, t_3]$.  Denote the collection of all real three-dimensional vectors by $\mathbb V$.

In general, a quaternion $\tilde q = [q_0, q_1, q_2, q_3]$ is a real four dimensional vector.   We distinguish it by the tilde symbol.   If $q_0 = 0$, then $\tilde q$ is called a vector quaternion, or an imaginary quaternion. Thus£¬ from $q = [q_1, q_2, q_3] \in \mathbb V$ we have a vector quaternion $\tilde q = [0, q]$ and vice versa.  Denote $q_0 = Re(\tilde q)$ as the real part of $\tilde q$, and $q = [q_1, q_2, q_3] = Im(\tilde q)$ as the imaginary part of $\tilde q$. The conjugate of $\tilde q$ is defined as $\tilde q^* = [q_0, -q_1, -q_2, -q_3]$.  The magnitude of $\tilde q$ is $|\tilde q| = \sqrt{q_0^2 + q_1^2+q_2^2+q_3^2}$.   The product of $\tilde q$ with a real number $\alpha$ is $\alpha \tilde q = [\alpha q_0, \alpha q_1, \alpha q_2, \alpha q_3]$.   The sum of two quaternions $\tilde p = [p_0, p_1, p_2, p_3]$ and  $\tilde q = [q_0, q_1, q_2, q_3]$ is $\tilde p + \tilde q = [p_0+q_0, p_1+q_1, p_2+q_2, p_3+q_3]$.  The product of $\tilde p$ and $\tilde q$ is $\tilde p\tilde q =
[p_0q_0-p_1q_1-p_2q_2-p_3q_3, p_0q_1+p_1q_0+p_2q_3-p_3q_2,$\\ $p_0q_2+p_2q_0-p_1q_3+p_3q_1,p_0q_3+p_3q_0+p_2q_3-p_3q_2]$.   The multiplication of quaternions is noncommutative.   The quaternion $\tilde 0 = [0, 0, 0, 0]$ plays the role of zero in the quaternion addition.  The quaternion $\tilde 1 = [1, 0, 0, 0]$ plays the role of the identity in the quaternion multiplication.  We have $\tilde q\tilde 1 = \tilde 1\tilde q = \tilde q$ and $\tilde q \tilde q^* = \tilde q^*\tilde q = |\tilde q|^2$ for any quaternion $\tilde q$.  A quaternion $\tilde q$ has an inverse $\tilde q^{-1}$ if and only if $\tilde q \not = \tilde 0$.  In this case, we have $\tilde q^{-1} = {\tilde q^* \over |\tilde q|^2}$.  Denote the collection of all quaternions by $\mathbb Q$.  We see that $\mathbb Q$ is nothing else but the collection of real four-dimensional vectors, equipped with the quaternion multiplication.

For two quaternions $\tilde p, \tilde q \in \mathbb Q$, we always have $(\tilde p\tilde q)^* = \tilde q^*\tilde p^*$.   Define the inner product of $\tilde p$ and $\tilde q$ as
$$\langle \tilde p, \tilde q \rangle = Re(\tilde p\tilde q^* + \tilde q\tilde p^*).$$
Then we have $\langle \tilde p, \tilde q \rangle = \langle \tilde q, \tilde p \rangle$.  Note that we always have $Im(\tilde p\tilde q^* + \tilde q\tilde p^*) = 0$.  If
the inner product of $\tilde p$ and $\tilde q$ is $0$, then we say that $\tilde p$ and $\tilde q$ are orthogonal.

A quaternion $\tilde q = [q_0, q_1, q_2, q_3] \in \mathbb Q$ is called a unit quaternion if $|\tilde q| = 1$. The set of unit quaternions is denoted as $\mathbb U$.  For $\tilde q \in \mathbb U$, we have
  $(\tilde q)^{-1} = \tilde q^*$.  A nonzero rotation vector $r$ denotes the frame rotation about a unit axis $l = {r \over \|r\|_2}$ with angle $0 \le \theta < 2\pi$. Here, $\theta = \|r\|_2$ mod$(2\pi)$.
We define a mapping $U : \mathbb V \to \mathbb U$ by
\begin{equation} \label{eq1}
U(r) = \left\{ \begin{aligned} \left[ \cos {\|r\|_2 \over 2}, {r \over \|r\|_2} \sin {\|r\|_2 \over 2} \right], & \ {\rm if}\  r \not = 0, \\
\tilde 1, &  \ {\rm if}\ r = 0.
\end{aligned} \right.
\end{equation}
We call $U$ the {\bf UQ operator}.
We further define a mapping $R : \mathbb U \to \mathbb V$ by
\begin{equation} \label{eq1.1}
R(\tilde q) = \left\{ \begin{aligned} {2 \cos^{-1} q_0 \over \sqrt{q_1^2+q_2^2+q_3^2}}[q_1, q_2, q_3], & \ {\rm if}\  q_0^2 \not = 1, \\
0, &  \ {\rm otherwise},
\end{aligned} \right.
\end{equation}
where $\tilde q = [q_0, q_1, q_2, q_3] \in \mathbb U$.  We call $R$ the {\bf rotation operator}.

By direct calculation, we have the following proposition.

\begin{Prop} \label{p0}
Let
$$\mathbb S = \{ r \in \mathbb V : \|r\|_2^2 = r_1^2 + r_2^2 + r_3^2 < 4\pi^2 \}.$$
Then for any $\tilde q \in \mathbb U \setminus \{ -\tilde 1 \}$, we have $U(R(\tilde q)) = \tilde q$, and for any $r \in \mathbb S$, we have $R(U(r)) = r$.
\end{Prop}
A rotation vector in $\mathbb S$ is called a {\bf basic rotation vector}.

{\bf Remark}  The unit quaternion $-\tilde 1$ is somewhat exceptional.   We have $U(R(-\tilde 1)) = \tilde 1 \not = -\tilde 1$.   The following proposition also indicates the special feature of $-\tilde 1$.

\begin{Prop}
Let $\tilde q = [q_0, q_1, q_2, q_3] \in \mathbb U$.   Then
\begin{equation} \label{aa1}
\tilde q^2 = [2q_0^2 - 1, 2q_0q_1, 2q_0q_2, 2q_0q_3].
\end{equation}
If furthermore $\tilde q$ is a unit imaginary quaternion, then
\begin{equation} \label{aa2}
\tilde q^2 = - \tilde 1.
\end{equation}
\end{Prop}
\begin{proof}
For $\tilde q = [q_0, q_1, q_2, q_3] \in \mathbb Q$, by quaternion multiplication, we have
$$\tilde q^2 = [q_0^2 - q_1^2-q_2^2-q_3^2, 2q_0q_1, 2q_0q_2, 2q_0q_3].$$
For $\tilde q = [q_0, q_1, q_2, q_3] \in \mathbb U$, we have $q_0^2+q_1^2+q_3^2+q_4^2=1$.  Thus, we have (\ref{aa1}).   If furthermore $\tilde q$ is imaginary, then we have $q_0=0$.  This leads to (\ref{aa2}).
\end{proof}



\bigskip

The set of all dual quaternions is denoted by $\hat {\mathbb Q}$. A dual quaternion $\hat q \in \hat {\mathbb Q}$ has the form
$\hat q = \tilde q + \tilde q_d\epsilon$, where $\tilde q, \tilde q_d \in \mathbb Q$ are quaternions, $\epsilon$ is the infinitesimal unit, satisfying $\epsilon^2 = 0$.   The quaternion $\tilde q$ is called the standard part of $\hat q$, and the quaternion $\tilde q_d$ is called the dual part of $\hat q$.  The infinitesimal unit $\epsilon$ is commutative with quaternions.    Based on this, we have addition and multiplication of dual quaternions, i.e., for two dual quaternions $\hat p = \tilde p + \tilde p_d\epsilon$
and $\hat q = \tilde q + \tilde q_d\epsilon$, we have
$$\hat p + \hat q = (\tilde p  + \tilde q) + (\tilde p_d + \tilde q_d)\epsilon$$
and
$$\hat p\hat q = \tilde p\tilde q + (\tilde p\tilde q_d + \tilde p_d\tilde q)\epsilon.$$
Thus, the identity of dual quaternions is $\hat 1 = \tilde 1 + \tilde 0\epsilon$, and the zero of dual quaternions is $\hat 0 = \tilde 0 + \tilde 0\epsilon$.

The conjugate of $\hat q = \tilde q + \tilde q_d\epsilon$ is $\hat q^* = \tilde q^* + \tilde q_d^*\epsilon$.
The magnitude of the dual quaternion $\hat q$ is
\begin{equation} \label{magnitude}
|\hat q| := \left\{ \begin{aligned} |\tilde q| + {\langle \tilde q, \tilde q_d \rangle \over 2|\tilde q|}\epsilon, & \ {\rm if}\  \tilde q \not = \tilde 0, \\
|\tilde q_d|\epsilon, &  \ {\rm otherwise}.
\end{aligned} \right.
\end{equation}
This is a dual number.  A dual number $\hat \alpha$ has the form $\hat \alpha = \alpha + \alpha_d \epsilon$, where $\alpha$ and $\alpha_d$ are real numbers.    Thus, a dual quaternion $\hat q = \tilde q + \tilde q_d\epsilon$ is a unit dual quaternion if and only if its standard part $\tilde q$ is a unit quaternion, i.e., $|\tilde q|=1$, and $\tilde q$ and $\tilde q_d$ are orthogonal, i.e.,
\begin{equation} \label{eq3}
\tilde q\tilde q_d^*+\tilde q_d \tilde q^* = \tilde 0.
\end{equation}
Denote the collection of all unit dual quaternions as $\hat {\mathbb U}$.   For $\hat q \in \hat {\mathbb U}$, we have $(\hat q)^{-1} = \hat q^*$.

\begin{Prop} \label{p1}
A dual quaternion $\hat q = \tilde q + \tilde q_d\epsilon$ is a unit dual quaternion if and only if $\tilde q$ is a unit quaternion and
\begin{equation} \label{eq4}
\tilde q_d = {1 \over 2}\tilde q\tilde t,
\end{equation}
where $\tilde t$ is a vector quaternion.
\end{Prop}
\begin{proof}
As $\tilde q$ is a unit quaternion, we may always write $\tilde q_d$ in the form of (\ref{eq4}).   We now have
$$\tilde q\tilde q_d^*+\tilde q_d \tilde q^* = {1 \over 2}\tilde q (\tilde t^* + \tilde t)\tilde q^*.$$
This shows that (\ref{eq3}) holds if and only if $\tilde t^* + \tilde t = \tilde 0$, i.e., $\tilde t$ is a vector quaternion.
\end{proof}
We call a real six-dimensional vector $x = [r, t]$ a {\bf motion}, where $r \in \mathbb V$ is called the rotation part of $x$,  and $t \in \mathbb V$ is called the translation part of $x$, respectively.    Denote the set of all motion vectors by $\mathbb M$.

\begin{Thm} \label{t1}
From a motion $x = [r, t] \in \mathbb M$, we have a unit dual quaternion
\begin{equation} \label{eq5}
\hat U(x) = U(r) + {\epsilon \over 2}U(r)\tilde t \in {\hat {\mathbb U}},
\end{equation}
where $U(r)$ is given by (\ref{eq1}), $\tilde t = [0, t]$.

On the other hand, from a unit dual quaternion $\hat q = \tilde q + \tilde q_d\epsilon$, as $\tilde q$ is a unit quaternion, we have $r = R(\tilde q) \in \mathbb S$.  Let
\begin{equation} \label{eq6}
\tilde T(\hat q) = 2(\tilde q)^{-1}\tilde q_d.
\end{equation}
Then $\tilde T(\hat q)$ is a vector quaternion, i.e.,  $\tilde T(\hat q) = [0, T(\hat q)]$, $T(\hat q) \in \mathbb V$.  Now£¬ let
\begin{equation} \label{eq10}
M(\hat q) = [R(\hat q), T(\hat q)].
\end{equation}
Then we have
\begin{equation} \label{eq7}
M(\hat q) \in \mathbb S \times \mathbb V \subset \mathbb M.
\end{equation}
Furthermore, for any $x \in \mathbb S \times \mathbb V$, we have $M(\hat U(x)) = x$, and for any $\hat q = \tilde q + \tilde q\epsilon \in \hat{\mathbb U}$ with $\tilde q \not = -\tilde 1$, we have $\hat U(M(\hat q)) = \hat q$.
\end{Thm}
\begin{proof}  Suppose that we have a motion $x = [r, t]$.   By the discussion before, $U(r)$, given by (\ref{eq1}), is a unit quaternion.
Let $\tilde q_d(x) = {1 \over 2}U(r)\tilde t$.    Then we have $\tilde t + \tilde t^* = \tilde 0$, i.e., $\tilde t$ is a vector quaternion.  By Proposition \ref{p1},
$\hat q = \hat U(x)$ is a unit dual quaternion.

On the other hand, assume that we have a unit dual quaternion $\hat q = \tilde q + \tilde q_d\epsilon$.   By the discussion before, $\tilde q$ is a unit quaternion, and $R(\tilde q) \in \mathbb S$.   Let $\tilde T(\hat q)$ be defined by (\ref{eq6}). By Proposition \ref{p1}, $\tilde T(\hat q)$ is a vector quaternion, i.e., $\tilde T(\hat q) = [0, T(\hat q)]$.  By (\ref{eq10}) and Proposition \ref{p0},  we have (\ref{eq7}).

The last conclusion follows from Proposition \ref{p0} and direct calculation.
\end{proof}

In (\ref{eq5}), we may also denote
\begin{equation} \label{eq8}
U_d(x) = {1 \over 2}U(r)\tilde t(x).
\end{equation}
Then we have
\begin{equation} \label{eq9}
\hat U(x) = U(r) + U_d(x)\epsilon.
\end{equation}

Call $M: {\hat {\mathbb U}} \to \mathbb M$ the {\bf motion operator}, $T: {\hat {\mathbb U}} \to \mathbb V$ the {\bf translation operator}, $\hat U: \mathbb M \to  {\hat {\mathbb U}}$ the {\bf UDQ operator}, and  $\tilde U_d: \mathbb M \to  {\tilde {\mathbb Q}}$ the {\bf dual quaternion operator}.



For a motion $x = [r, t] = [r_1, r_2, r_3, t_1, t_2, t_3] \in \mathbb M$, define its magnitude as
\begin{equation}
|x| = \sqrt{|r_1|^2 + |r_2|^2 + |r_3|^2 + \sigma|t_1|^2 + \sigma|t_2|^2 + \sigma|t_3|^2},
\end{equation}
where $\sigma$ is a positive number.   

\section{Motion Optimization}

A motion vector $\vx = [x^{(1)}, \cdots, x^{(n)}]$ is an $n$-component vector such that its $i$th components is a motion $x^{(i)} = \left[r^{(i)}, t^{(i)}\right] = \left[r_1^{(i)}, r_2^{(i)}, r_3^{(i)}, t_1^{(i)}, t_2^{(i)}, t_3^{(i)}\right] \in \mathbb M$, for $i = 1, \cdots, n$.  In a certain sense, we may also regard $\vx$ as a $6n$-dimensional real vector.  We use small bold letters such as $\vx$ to denote motion vectors, and denote the space of $n$-component motion vectors by ${\mathbb M}^n$.  The norm in ${\mathbb M}^n$ is defined by
\begin{equation}
\|\vx\| = \sqrt{\sum_{i=1}^n \left|x^{(i)}\right|^2} = \sqrt{\sum_{i=1}^n \left|r_1^{(i)}\right|^2 + \left|r_2^{(i)}\right|^2 + \left|r_3^{(i)}\right|^2 + \sigma\left|t_1^{(i)}\right|^2 + \sigma\left|t_2^{(i)}\right|^2 + \sigma\left|t_3^{(i)}\right|^2}.
\end{equation}
It is not difficult to show that it is a norm.

Suppose that we have $\vz : {\mathbb M}^n \to {\mathbb M}^m$. A {\bf motion optimization problem} can be formulated as
\begin{equation}
\min \left\{ {1 \over 2}\| \vz(\vx)\|^2 : \vx \in {\mathbb M}^n \right\}.
\end{equation}
This is a $6n$-dimensional unconstrained optimization problem.

\medskip

 {\bf Example 1}  Consider the 1989 Shiu and Ahmad \cite{SA89} and Tsai and Lenz \cite{TL89} hand-eye calibration model.  We have $n = 1$.   In \cite{Qi22}, it is formulated as unit dual quaternion equations
 \begin{equation} \label{eq12}
 \hat a^{(i)}\hat q = \hat q\hat b^{(i)},
 \end{equation}
 where, $\hat a^{(i)}, \hat b^{(i)}, \hat q \in {\hat {\mathbb U}}$, for $i = 1, \cdots, m$.    Here, $\hat q$ is the transformation unit dual quaternion from the camera (eye) to the gripper (hand), $\hat a^{(i)}, \hat b^{(i)}$, for $i = 1, \cdots, m$, are some data unit dual quaternions from experiments.   The aim is to find the best unit dual quaternion $\hat q$ to satisfy (\ref{eq12}).  Then, let
 \begin{equation}
 \hat q = \hat U(x),
 \end{equation}
 \begin{equation} \label{eq16}
 z_i = M\left(\hat a^{(i)}\hat q \left(\hat p\hat b^{(i)}\right)^{-1}\right),
 \end{equation}
 and $\vz = [z_1, \cdots, z_m] \in {\mathbb M}^m$.
We have the following motion optimization problem
\begin{equation} \label{eq16}
\min \left\{ {1 \over 2}\| \vz(x)\|^2 : x \in \mathbb M \right\}
\end{equation}
for this hand-eye calibration model.   This is a $6$-dimensional unconstrained optimization problem.

\medskip

  {\bf Example 2}  Consider the 1994 Zhuang, Roth and Sudhaker \cite{ZRS94} hand-eye calibration model.  We have $n = 2$.   In \cite{Qi22}, it is formulated as unit dual quaternion equations
 \begin{equation} \label{eq18}
 \hat a^{(i)}\hat q = \hat p\hat b^{(i)},
 \end{equation}
 where, $\hat a^{(i)}, \hat b^{(i)}, \hat p, \hat q \in {\hat {\mathbb U}}$, for $i = 1, \cdots, m$.    Here, $\hat p$ is the transformation unit dual quaternion from the world coordinate system to the robot base.  The aim is to find the best unit dual quaternions $\hat p$ and $\hat q$ to satisfy (\ref{eq18}).  Then, let
 \begin{equation}
 \hat q = \hat U\left(x^{(1)}\right), \hat p = \hat U\left(x^{(2)}\right), \vx = \left[x^{(1)}, x^{(2)}\right],
 \end{equation}
 \begin{equation}
 z_i = M\left(\hat a^{(i)}\hat q \left(\hat p\hat b^{(i)}\right)^{-1}\right),
 \end{equation}
 and $\vz = [z_1, \cdots, z_m] \in {\mathbb M}^m$.
We have the following motion optimization problem
\begin{equation}  \label{eq20}
\min \left\{ {1 \over 2}\| \vz(\vx)\|^2 : \vx \in {\mathbb M}^2 \right\}
\end{equation}
for this hand-eye calibration model.   This is a $12$-dimensional unconstrained optimization problem.

\medskip

  {\bf Example 3} Consider the simultaneous localization and mapping (SLAM) problem.   We have a directed graph $G = (V, E)$ \cite{CTDD15}, where each vertex $i \in V$ corresponds to a robot pose $\hat p_i \in {\hat {\mathbb U}}$ for $i = 1, \cdots, n$, and each directed edge (arc) $(i, j) \in E$ corresponds to a relative measurement $\hat q_{ij}$, also a unit dual quaternion.  There are $m$ directed edges in $E$.   In \cite{Qi22}, a dual quaternion optimization problem is formulated to find the best $\hat p_i$ for $i = 1, \cdots, n$, to satisfy
  \begin{equation}
   \hat q_{ij} = \hat p_i^* \hat p_j
  \end{equation}
  for $(i, j) \in E$.  We now formulate it as a motion optimization problem.  Let
  \begin{equation}
  x^{(i)} = M(\hat p_i)
  \end{equation}
  for $i = 1, \cdots, n$.   Let $\vx = [x^{(1)}, \cdots, x^{(n)}]$.   Then $\vx \in {\mathbb M}^n$ and
  \begin{equation}
  \hat p_i = \hat U(x^{(i)})
  \end{equation}
   for $i = 1, \cdots, n$.   Let
   \begin{equation}
   z_{ij} = M\left(\hat p_i^* \hat p_j(\hat q_{ij})^{-1}\right)
   \end{equation}
  for $(i, j) \in E$ and $\vz = [z_{ij}, : (i, j) \in E] \in {\mathbb M}^m$.   Then we have the following motion optimization problem
\begin{equation} \label{eq25}
\min \left\{ {1 \over 2}\| \vz(\vx)\|^2 : \vx \in {\mathbb M}^n \right\}
\end{equation}
for the SLAM problem.   This is a $6n$-dimensional unconstrained optimization problem.

\section{Motion Operator and UDQ Operator}

The formulation of the motion optimization problems (\ref{eq16}), (\ref{eq20}) and (\ref{eq25}) highly relies on the motion operator $M$ and the UDQ operator $\hat U$.   Thus, we need to know more properties of $M$, $\hat U$ as well as the rotation operator $R$, the translation operator $T$, the UQ operator $U$ and the dual quaternion operator $U_d$.

In some sense, the motion operator $M$ and the rotation operator $R$ are close to the logarithm operators of unit dual quaternions and unit quaternions in the literature \cite{WHYZ12, KKS96}.   Let $\tilde q$ be a unit quaternion, and $r \in \mathbb M$ be the corresponding rotation.  Then we have
$$\ln \tilde q = \left[0, {1 \over 2}R(\tilde q)\right].$$
Note that in general, we have $\ln \tilde p \tilde q \not = \ln \tilde p + \ln \tilde q$.  Also see \cite{QWL22}.

In fact, for $\tilde p, \tilde q \in {\tilde {\mathbb U}}$, in general, we should have
$$R(\tilde p\tilde q) \not = R(\tilde p) + R(\tilde q).$$
Similarly, for $\hat p, \hat q \in {\hat {\mathbb U}}$, in general, we should have
$$M(\hat p\hat q) \not = M(\hat p) + M(\hat q).$$

Then, what are the properties of $M, \hat U, R, T, U$ and $U_d$?

We consider some algebraic properties of these operators.

\begin{Thm}
We have the following conclusions:

(a) $R(\tilde 1) = 0, U(0) = \tilde 1, M(\hat 1) = \tilde 0, \hat U(\tilde 0) = \hat 1$.

(b) For any $r \in \mathbb V$, $U(-r) = \left(U(r)\right)^{-1}$, i.e., for any $\tilde q \in \mathbb U$, $R(\tilde q^{-1}) = - R(\tilde q)$.

(c) For any $x \in \mathbb M$, $\hat U(-x) = \left(\hat U(x)\right)^{-1}$, i.e., for any $\hat q \in \hat {\mathbb U}$, $M(\hat q^{-1}) = - M(\hat q)$.

(d) For $\tilde q \in \mathbb U$, let $R(\tilde q) = \theta l$ such that $l$ is a unit axis and $0 \le \theta < 2\pi$.   Then
\begin{equation} \label{ae1}
R(\tilde q^2) =  \left\{ \begin{aligned} 2R(\tilde q), & \ {\rm if}\  0 \le \theta < \pi, \\
2(\theta-\pi)l, &  \ {\rm otherwise}.
\end{aligned} \right.
\end{equation}
\end{Thm}
\begin{proof}
By definitions and direct calculation, we have conclusion (a).

From (\ref{eq1}), for any $r \in \mathbb V$, we have
$$U(-r) = \left[\cos \left({\theta \over 2}\right), -l\sin \left({\theta \over 2}\right)\right] = \tilde q^* = \left(\tilde q\right)^{-1} = \left(U(r)\right)^{-1}.$$
This proves (b).

From (\ref{eq8}) and (\ref{eq9}), for any $x \in \mathbb M$, we have
$$\hat U(-x) = U(-r(x)) + {1 \over 2}U(-r(x))\tilde t(-x)\epsilon.$$
Then, by Conclusion (b) of this theorem, we further have
$$\hat U(-x) = \left(U(r(x))\right)^{-1} - {1 \over 2}\left(U(r(x))\right)^{-1}\tilde t(x)\epsilon.$$
Then by (\ref{eq9}) and this, we have
$$\hat U(x)\hat U(-x) = \hat 1,$$
i.e.,
$$\hat U(-x) = \left(\hat U(x)\right)^{-1}.$$
This proves (c).

By (\ref{eq1.1}) and (\ref{aa1}), we have (\ref{ae1}), i.e., (d).
\end{proof}

This theorem has a strong geometrical meaning.   Conclusion (a) means the idle rotation and the idle motion. Conclusion (b) means an inverse rotation, while conclusion (c) means an inverse motion.  Conclusion (d) means a double rotation.   It is possible to extend Conclusion (d) to the case of $R(\tilde q^k)$ for an integer $k \ge 3$, but it does not work for a double motion.

\section{Final Remarks}

In this paper, we introduced motions, the motion operator, the UDQ operator and motion optimization.
The hand-eye calibration problem and the SLAM problem were formulated as motion optimization problems, which are unconstrained optimization problems.  Properties of the motion operator and the UDQ operator were studied.

There are several issues which we will further study in future.   First, we plan to study the existence of global optimizers in ${\mathbb S}^n \times {\mathbb V}^n$.  Second, we plan to study some further properties of the motion operator $M$ and the UDQ operator $\hat U$.   Third, we plan to pursue practical numerical methods to solve these motion optimization problems.

\bigskip

{\bf Acknowledgment}  I am thankful to Zhongming Chen, Chen Ling, Ziyan Luo and Xiangke Wang for their comments.


\bigskip




\begin{thebibliography}{10}
\bibitem{Qi22} L. Qi, ``Standard dual quaternion optimization and its application in hand-eye calibration'', {\sl Communications on Applied Mathematics and Computation} DOI£º 10.1007/s42967-022-00213-1.
\end{thebibliography}
\end{document}